\documentclass[12pt]{article}
\usepackage[latin1]{inputenc}
\usepackage{amsmath}
\usepackage{amsfonts}
\usepackage{amssymb}
\usepackage{amsthm}

\usepackage{hyperref}

\usepackage[all]{xy}
\usepackage[usenames,dvipsnames]{color}
\usepackage[dvips]{graphicx}

\newtheorem{teo}{Theorem}[section]
\newtheorem{df}[teo]{Definition}
\newtheorem{pro}[teo]{Proposition}
\newtheorem{cor}[teo]{Corollary}
\newtheorem{lem}[teo]{Lemma}

\newtheorem{ex}[teo]{Examples}

\long\def\symbolfootnote[#1]#2{\begingroup
\def\thefootnote{\fnsymbol{footnote}}\footnote[#1]{#2}\endgroup} 

\title{The leaves of the Fatou set accumulate on the leaves of the Julia set.}
\author{Nicolas Hussenot Desenonges}

\begin{document}
\maketitle

\begin{abstract}
In $2001$, E. Ghys, X. G\'{o}mez-Mont and J. Saludes defined in \cite{GGS} the Fatou and Julia components of transversely holomorphic foliations on compact manifolds. It is a partition of the manifold in two saturated sets: the Fatou set which is open and represents the non-chaotic part of the foliation and its complementary set, the Julia set. Using the Brownian motion transverse to the leaves, it is proved that, if the foliation is taut and if $F$ is a wandering component of the Fatou set, then almost every point of the topological boundary $\partial{F}$ (almost for any harmonic measure on $\partial{F}$) is a limit point of each leaf of $F$.
\end{abstract}

\section{Introduction}

The theory of conformal dynamical systems on the Riemann sphere $\mathbb{C}\mathbb{P}^1$ consists in two main subjects of study: the dynamics of Kleinian groups and the dynamics of iteration of a rational map on the sphere. Each of these theories leads to a partition of $\mathbb{C}\mathbb{P}^1$ in two invariant sets: \textit{domain of discontinuity/limit set} for Kleinian groups and \textit{Fatou set/Julia set} for iteration of a rational map. The famous Sullivan's dictionary brings light  links between these two theories. The density of orbits of the Julia set (resp. limit set) is one of the first analogies. More precisely, we have:
\begin{enumerate}
\item if $\Gamma$ is a non elementary Kleinian group, then for all $x\in\mathbb{C}\mathbb{P}^1$, the orbit $\Gamma.x$ accumulates on every point of the limit set (see for example \cite{Ma}).
\item if $f$ is a rational map of degree two or more, then for any $z\in Julia(f)$, the set of iterated pre-images of $f$ is dense in $Julia(f)$ (see for example \cite[Theorem3]{Mi}).
\end{enumerate}

In \cite{GGS}, Etienne Ghys, Xavier G\'{o}mez-Mont and Jordi Saludes consider a compact manifold with a transversely holomorphic foliation and develop a theory analoguous to the above mentioned theories. More precisely, they manage to define dynamically a partition of the manifold in two saturated sets called Julia set and Fatou set.
By analogy with the properties cited above concerning Kleinian groups and iteration of rational maps, the following questions arise naturally: If $L$ is a leaf contained in a connected component $J$ of the Julia set, do we have $\overline{L}=J$? Does the Julia set play the role of an attractor for the leaves of the  Fatou set? Using probabilistic tools (Brownian motion), we answer positively to the second question. More precisely, we prove that if the foliation is taut and if $F$ is a wandering component of the Fatou set, then every leaf of $F$ accumulates on almost every point of the topological boundary $\partial F\subset Julia$. The term almost in the previous sentence is with respect to the harmonic measure on $\partial F$ (i.e. the exit measure of $F$ for a Brownian motion starting at an arbitrary point in $F$).

The Brownian motion has already been used in order to study the dynamics of the leaves of foliations: this is Lucy Garnett's theory of harmonic measures \cite{Ga}. B. Deroin and V. Kleptsin used this theory in the particular case of transversely holomorphic foliations \cite{DK}. They proved that the following dichotomy holds: either there exists an invariant transverse measure or there exists a finite number of minimal sets $\mathcal{M}_1,...,\mathcal{M}_k$ equipped with probability measures $\nu_1,...,\nu_k$ such that for every point $x$ in the manifold and almost every leafwise Brownian path $\omega$ starting at $x$, the path $\omega$ tends to one of the $\mathcal{M}_j$ and is distributed with respect to $\nu_j$ in the sense that:
$$\lim\limits_{t\to \infty}\frac{1}{t}\omega_{*}leb_{[0,t]}=\nu_j,$$
where $leb_{[0,t]}$ is the Lebesgue measure on the interval [0,t].

In Garnett's theory, the Brownian motion is tangent to the leaves of the foliation. In this paper, we use a totally different approach of the Brownian motion in the sense that the Brownian motion is essantially transverse to the foliation.

\paragraph{Fatou and Julia components of transversely holomorphic foliations. }
Let $M$ be a compact, connected manifold of real dimension $d+2$ endowed with a transversely holomorphic foliation $\mathcal{F}$. We can define a partition of the manifold in two saturated sets: the Fatou set is defined as the set of points $x$ in $M$ such that there exists a basic normal vector field (i.e. a section of the normal bundle to the foliation constant along the leaves) which does not vanish at $x$ (we will come back later to the demanded regularity for the vector field). The Fatou set is then an open set. Its complementary set is a closed set called Julia set. Given a connected component  $F$ of the Fatou set, it is easy to prove (integrating the basic normal vector fields), that the group of homeomorphisms of $M$ which preserve the foliation acts transitively on $F$. Using this homogeneity property and by analogy with Molino's theory \cite{Mo}, Ghys, G\'{o}mez-Mont and Saludes proved that there are only three exclusive cases for a connected component $F$ of the Fatou set (cf theorem \ref{theoGGS1}): $F$ is a wandering component (i.e. all the leaves of $\mathcal{F}_{|F}$ are closed in $F$) or $F$ is a dense component (all the leaves of $\mathcal{F}_{|F}$ are dense in $F$) or $F$ is a semi-wandering component (the closure of the leaves of $\mathcal{F}_{|F}$ form a real codimension $1$ foliation of $F$). 
In \cite{GGS}, the authors give a complete description of the Fatou set studying in details each of the three previous cases. In particular, they prove that if $F$ is a wandering component of the Fatou set, then the leaf space $F/\mathcal{F}_{|F}$ is a Hausdorff Riemann surface of finite type (it is the analoguous of Ahlfors' finiteness theorem \cite{Ah}).

Let us mention that an alternative definition of the Fatou and Julia sets of transversely holomorphic foliations has been given by T. Asuke in \cite{As}.

\paragraph {Statement of the theorem. }
Let $M$ be a compact manifold endowed with a transversely holomorphic foliation $\mathcal{F}$. Endow $M$ with a complete Riemannian metric $g$. Then, we can define the Brownian motion on $(M,g)$: it is the diffusion process $(B_t)_{t\geq 0}$ associated to the Laplace-Beltrami operator on $(M,g)$. It is defined on the family of probability spaces $(\Omega_x,\mathbb{P}_x)_{x\in M}$.
Let $F$ be a wandering connected component of the Fatou set. We define the family of harmonic measures $(\nu_x)_{x\in F}$ by:
$$\nu_x(A)=\mathbb{P}_x(B_T\in A),$$
where $T=\inf\{t\in[0;+\infty] \text{ s.t. } B_t\in \partial F\}$ is the hitting time of $\partial F$ and $A$ is any Borel set in $\partial F$. The manifold $M$ being compact, if we make another choice for the metric or another choice for the point $x$, we get a new measure on $\partial F$ which is equivalent to the previous one. So, this makes sense to talk about the class $\nu$ of harmonic measures on $\partial F$.
We would like to push forward the Brownian motion $(B_t)_{t\in[0;T[}$ in $F$ on the leaf space $F/\mathcal{F}_{|F}$ and get a time-changed Brownian motion. For doing this, we will need the following assumption on the manifold $M$.

\begin{df}
A foliated manifold is said to be taut if there exists a metric on $M$ for which all the leaves are minimal submanifolds.
\end{df}
 
The main theorem of this paper is the following:

\begin{teo}\label{theoperso1}
Let $M$ be a compact, connected manifold endowed with a transversely holomorphic foliation $\mathcal{F}$. Suppose that $\mathcal{F}$ is taut.
Let $F$ be a wandering component of the Fatou set. Assume that $\nu(\partial F)=1$ (i.e. almost every Brownian path starting in $F$ hits the Julia set in finite time). Then, each leaf of $F$ accumulates on $\nu$-almost every point of $\partial F$.
\end{teo}

\textbf{Remark. }The tautness hypothesis is essantial in our proof but the theorem may still be true without this hypothesis. Nevertheless, this hypothesis is not so strong: for example a foliation without invariant transverse measure is taut (see \cite{S2} or \cite{CC1} for more details about taut foliations).

\paragraph{Analogy with Kleinian groups.} Let $\Sigma$ be a compact Riemann surface and $\rho:\pi_1(\Sigma)\rightarrow PSL(2,\mathbb{C})$ be a morphism from the fondamental group of $\Sigma$ to the group of biholomorphisms of the Riemann sphere $\mathbb{C}\mathbb{P}^1$. Suspending the representation $\rho$, one gets a compact manifold $M_{\rho}$ which fibers over $\Sigma$ (fibers are copies of $\mathbb{C}\mathbb{P}^1$) and which is endowed with a transversely holomorphic foliation. The dynamics of this foliation correspond to the dynamics of the action of the monodromy group $\Gamma:=\rho(\pi_1(\Sigma))$ on a fiber. Basic normal vector fields on $M_{\rho}$ correspond to $\Gamma$-invariant vector fields in the fiber.

From now on, assume $\Gamma$ is a Kleinian group (i.e. a discret subgroup of $PSL(2,\mathbb{C})$). Then we have the classical partition: $\mathbb{C}\mathbb{P}^1=\Omega(\Gamma)\cup\Lambda(\Gamma)$. By definition the domain of discontinity $\Omega(\Gamma)$ is the set of points $z\in\mathbb{C}\mathbb{P}^1$ such that there exists an open set $U$ containing $z$ and satisfying the following: $\gamma U\cap U=\varnothing$ for all but a finite number of $\gamma\in\Gamma$. And the limit set $\Lambda(\Gamma)$ is the complementary set of $\Omega(\Gamma)$. Assume for simplicity that $\Gamma$ is torsion free and $\Omega(\Gamma)\neq\varnothing$. As it is explained in \cite[example~8.4]{GGS}, the Julia set (resp. the Fatou set) of the suspended foliation corresponds to the saturated set of $\Lambda(\Gamma)$ (resp. $\Omega(\Gamma)$).
 
If $\Gamma$ is non elementary (i.e. if the limit set contains strictly more than two points), a classical property of the theory of Kleinian groups is the following: every point of the domain of discontinuity accumulates on every point of the limit set. The proof of this fact is very easy (see for example \cite{Ma}). We are going to give an idea of an alternative proof of this fact using conformal invariance of Brownian motion because it is this idea that will be generalized in the proof of the main theorem \ref{theoperso1}. Let $\Gamma$ be a Kleinian group, torsion free, non elementary and with $\Omega(\Gamma)\neq\varnothing$. Endow $\mathbb{C}\mathbb{P}^1$ with its spherical metric. Consider a Brownian motion $B_t$ in $\mathbb{C}\mathbb{P}^1$ starting at a point $x\in\Omega(\Gamma)$. Let $T:=\inf\{t\geq 0\text{ s.t. } B_t\in\Lambda(\Gamma)\}$. According to a theorem of Myrberg \cite{My} (see also \cite{Do}), the limit set of a non elementary Kleinian group has positive logarithmic capacity. So, almost surely $T<\infty$. So, the harmonic measure $\nu_x$ on $\Lambda(\Gamma)$ defined by $\nu_x(A)=\mathbb{P}_x(B_T\in A)$ is a probabilty measure with $supp(\nu_x)=\Lambda(\Gamma)$. The natural projection $p:\Omega(\Gamma)\rightarrow \Omega(\Gamma)/\Gamma$ is holomorphic. So, according to the celebrated Paul Lévy's theorem, the processus $p(B_t)_{t\in[0;T[}$ is a changed-time Brownian motion. This means that there is a time reparametrization $\sigma$ (i.e. $\sigma$ is a random strictly increasing map from $[0;T[$ to $\mathbb{R}^+$) such that the process $\left\{p(B_{\sigma^{-1}(s)}),s\in[0,\lim\limits_{t\to T}\sigma(t)[\right\}$ is a Brownian motion starting at $p(x)$. An easy topological argument (which will be generalized to our context in the proof of the main theorem) shows that necessarely $\lim\limits_{t\to T}\sigma(t)=\infty$. The Ahlfors' finitness theorem asserts that the quotient $\Omega(\Gamma)/\Gamma$ is a Riemann surface of finite type (i.e. a compact Riemann surface with a finite number of points deleted) \cite{Ah}. In particular, the Brownian motion is recurrent on $\Omega(\Gamma)/\Gamma$. So, if we take an arbitrary point $y$ in $\Omega(\Gamma)$ and $U_y$ an arbitrarily small open neighborhood of $y$, there exists arbitrarily large times $s$ such that $p(B_{\sigma^{-1}(s)})$ belongs to $p(U_y)$. Pushing backward by $p$, we deduce that there exist times $t$ arbitrarily close to $T$ such that $B_t$ belongs to $p^{-1}(p(U_y))=\Gamma.U_y$. So we have just proved that $\Gamma.U_y$ accumulates on $\nu_x$-almost every point of $\Lambda(\Gamma)$. The point $y$ and neighborhood $U_y$ being arbitrarily choosen, and $supp(\nu_x)$ being equal to the entire $\Lambda(\Gamma)$, one concludes (with a little work) that every point of the domain of discontinuity accumulates on every point of the limit set.

\paragraph{Questions. } The following questions are not answered in this paper:
\begin{enumerate} 
\item Is the class of harmonic measures non zero? In other words, does a Brownian motion starting at a point $x\in Fatou$ visits the Julia set? If not, our theorem does not say anything. Nevertheless, by analogy with the result of Makarov which asserts that the limit set of a non elementary Kleinian group is visited by Brownian motion, we conjecture that in presence of sufficiently rich dynamics for the foliation, the Julia set is visited by the Brownian motion. 
\item If the harmonic measure in $\partial F$ is non zero, what can be said about the support of the measure? Is it the full $\partial F$? What are the dynamics inside the Julia set? 
\end{enumerate}

\paragraph {Acknowledgements}
This paper comes from the first part of my Phd thesis \cite{Hu}. I am very grateful to my advisor Ga\"{e}l Meigniez for his precious help during all these years.

\section{Fatou and Julia components.}

In this first section, we are going to define the Fatou and Julia components of a transversely holomorphic foliation following the presentation of  \cite{GGS}. 

Let $M$ be a compact, connected manifold of real dimension $d+2$ endowed with a transversely holomorphic foliation $\mathcal{F}$. Such a foliation may be defined by an atlas $(U_i,\varphi_i,\gamma_{ij})$. The $U_i$ are open sets in $M$ covering $M$. The maps $\varphi_i:U_i\rightarrow \mathbb{C}$ are submersions with connected fibres and the maps $\gamma_{ij}:\varphi_j(U_i\cap U_j)\rightarrow \varphi_i(U_i\cap U_j)$ are biholomorphisms and satisfy $\varphi_i=\gamma_{ij}\circ\varphi_j$. 

Look at $(M,\mathcal{F})$ as a foliation of complex codimension $1$. Denote $TM$ the tangent bundle, $T\mathcal{F}$ the subbundle of $TM$ consisting of those vectors which are tangent to the leaves and $\nu^{1,0}$ the quotient bundle $TM/T\mathcal{F}$. One has the following exact sequence:
\begin{equation}\label{eq1}
0\longrightarrow T\mathcal{F}\longrightarrow TM \longrightarrow \nu^{1,0}\longrightarrow 0
\end{equation}

Let $E$ be a vector bundle over $M$.
We say that the germ of a section $X$ of $E$ at a point $x$ of $M$ has modulus of continuity $\epsilon \log\epsilon$ if there is a positive constant $C$, a coordinate chart $U$ containing $x$, and a trivialization of the bundle over $U$ such that for $x_1$, $x_2$ in $U$, we have:
\[| X(x_2)-X(x_1)| <-C| x_1-x_2| \log| x_1-x_2|\]
Denote $\mathcal{C}^{\epsilon \log\epsilon}(E)$ the sheaf of sections of the vector bundle $E$ which have a modulus of continuity $\epsilon \log\epsilon$. The sheaves of sections of (\ref{eq1}) of modulus of continuity $\epsilon \log\epsilon$ give rise to the following exact sequence of fine sheaves:
\[0\longrightarrow \mathcal{C}^{\epsilon \log\epsilon}(T\mathcal{F})\longrightarrow \mathcal{C}^{\epsilon \log\epsilon}(TM)\longrightarrow \mathcal{C}^{\epsilon \log\epsilon}(\nu^{1,0})\longrightarrow 0\]

The normal bundle $\nu^{1,0}$ is flat along the leaves. Indeed, it can be defined by the cocycle: $U_i\cap U_j\rightarrow \mathbb{C}^*$, $p\mapsto \gamma_{ji}'(\varphi_i(p))$. These maps are constant along the fibres of $\varphi_i$ so they are constant along the leaves. Any time that we have a bundle which is flat along the leaves, we can define the sections of this bundle which are constant along the leaves. We call these sections basic sections.

Likely as $\nu^{1,0}$, the bundle $\nu^{1,0}\otimes\nu^{0,1*}$ is flat along the leaves. Denote $L_{\mathcal{F}}^{\infty}(\nu^{1,0}\otimes\nu^{0,1*})$ the sheave of basic sections of this bundle which are essentialy bounded. Denote also $\mathcal{C}_{\mathcal{F}}(\nu^{1,0})$ the sheaf of continuous basic sections of the bundle $\nu^{1,0}$ satisfying: $\forall \sigma \in \mathcal{C}_{\mathcal{F}}(\nu^{1,0})$, $\overline{\partial}\sigma \in L_{\mathcal{F}}^{\infty}(\nu^{1,0}\otimes\nu^{0,1*})$.  Denote $\mathcal{C}_{\mathcal{F}}^{\epsilon \log\epsilon}(TM):=\pi^{-1}(\mathcal{C}_{\mathcal{F}}(\nu^{1,0}))$ (where $\pi$ is the projection  $\pi:\mathcal{C}^{\epsilon \log\epsilon}(TM)\longrightarrow \mathcal{C}^{\epsilon \log\epsilon}(\nu^{1,0})$). We have the following exact sequence:
\[0\longrightarrow \mathcal{C}^{\epsilon \log\epsilon}(T\mathcal{F})\longrightarrow \mathcal{C}_{\mathcal{F}}^{\epsilon \log\epsilon}(TM)\longrightarrow \mathcal{C}_{\mathcal{F}}(\nu^{1,0})\longrightarrow0\]

The first sheaf is a fine one. So, $H^1(M,\mathcal{C}^{\epsilon \log\epsilon}(T\mathcal{F}))=0$, which gives rise to the following exact sequence of global sections:

\[0\rightarrow H^0(M,\mathcal{C}^{\epsilon \log\epsilon}(T\mathcal{F}))\rightarrow H^0(M,\mathcal{C}_{\mathcal{F}}^{\epsilon \log\epsilon}(TM))\rightarrow H^0(M,\mathcal{C}_{\mathcal{F}}(\nu^{1,0}))\rightarrow0\]

This implies that we can lift any basic normal vector field $X \in H^0(M,\mathcal{C}_{\mathcal{F}}(\nu^{1,0}))$ to a vector field in $M$ with modulus of continuity $\epsilon \log\epsilon$. Such vector fields have the property to be uniquely integrable in the sense that the equation $x'=X(x)$ has a unique solution for a given initial condition, and so defines a local flow \cite{R}.

Summarizing, we have the following:

\begin{lem}\cite{GGS}
\begin{enumerate}
\item Any basic normal vector field $X \in H^0(M,\mathcal{C}_{\mathcal{F}}(\nu^{1,0}))$ can be lifted to a vector field in $H^0(M,\mathcal{C}_{\mathcal{F}}^{\epsilon \log\epsilon}(TM))$.
\item Any vector field $X \in H^0(M,\mathcal{C}_{\mathcal{F}}^{\epsilon \log\epsilon}(TM))$ gives rise to a global 1-parameter flow $\phi:M\times \mathbb{R}\rightarrow M$ preserving the foliation.
\end{enumerate}
\end{lem}

We can now define the Fatou and Julia sets of a transversely holomorphic foliated compact manifold $(M,\mathcal{F})$:

\begin{df}
\begin{itemize}
\item The Julia set of $(M,\mathcal{F})$ is the closed saturated set where all the elements of $H^0(M,\mathcal{C}_{\mathcal{F}}(\nu^{1,0}))$ vanish :
\[Julia(\mathcal{F})=\left\{x \in M \text{ s.t. } X(x)=0 \quad \forall X\in H^0(M,\mathcal{C}_{\mathcal{F}}(\nu^{1,0}))\right\}\]
\item The Fatou set of $(M,\mathcal{F})$ is the open and saturated set defined as the complement of the Julia set:
\[Fatou(\mathcal{F})=M-Julia(\mathcal{F})\]
\end{itemize}
\end{df}

Using the existence of non trivial basic normal vector fields at every point of the Fatou set, the authors prove the following homogeneity property:

\begin{pro}\cite{GGS}
Let $F_k$ be a connected component of the Fatou set. Given $x_1$ and $x_2$ two points in $F_k$, there is a $\mathcal{F}$-preserving homeomorphism of $M$ sending $x_1$ to $x_2$.
\end{pro}

\begin{proof}
First, for the tangent direction, given a point $x$ in $M$, it is easy to find $\mathcal{F}$ preserving homeomorphisms sending $x$ to any point located in the same leaf as $x$ (we just have to integrate vector fields tangent to the leaves).

For the transversal direction, we use the fact that if $x$ is in the Fatou set, then there exists $X \in  H^0(M,\mathcal{C}_{\mathcal{F}}(\nu^{1,0}))$ with $X(x)\neq 0$. Let $\tilde{X}$ and $\tilde{Y}$ in $H^0(M,\mathcal{C}_{\mathcal{F}}^{\epsilon\log\epsilon}(TM))$ lifting $X$ and $iX$ and consider the map:
\[\begin{array}{ccccc}
\phi & : & \mathbb{C}\times M & \to & M \\
& & (te^{i\theta},x) & \mapsto & \phi(te^{i\theta},x) \\
\end{array}\]
which associates to $(te^{i\theta},x)$ the point $\phi(te^{i\theta},x)$ of $M$ obtained by integrating the vector field $\tilde{X_{\theta}}:=\tilde{X}\cos{\theta}+\tilde{Y}\sin{\theta}$ at time $t$. As $\tilde{X_{\theta}}$ is a basic vector field, if we fix $te^{i\theta}$, the map $\phi(.,te^{i\theta})$ preserves $\mathcal{F}$. With this method, we get $\mathcal{F}$ preserving homeomorphisms sending $x$ to any point located in a transverse topological disc centered in $x$. 

A composition of maps of both types prove the assertion. 
\end{proof}

The last proposition is the key to prove the following theorem:

\begin{teo}\cite{GGS}\label{theoGGS1}
Let $(M,\mathcal{F})$ be a transversely holomorphic foliated compact manifold. Let $\mathcal{F}_k$ be the restriction of $\mathcal{F}$ to a connected component $F_k$ of $Fatou(\mathcal{F})$. Then, there are $3$ exclusive cases:
\begin{enumerate}
\item Wandering component: the leaves of $\mathcal{F}_k$ are closed in $F_k$.
\item Semi-wandering component: the closure of the leaves of $\mathcal{F}_k$ form a real codimension $1$ foliation of $F_k$ which has a structure of a fibre bundle over a $1$-dimensional manifold.
\item Dense component: all the leaves of $\mathcal{F}_k$ are dense in $F_k$.
\end{enumerate} 
\end{teo}

The transversely holomorphic structure of the foliation naturally endows the leaf space with a conformal structure. In general, this space is non Hausdorff. But, if one restricts the leaf space to a wandering component, one can prove the following anologuous of Ahlfors' finiteness theorem:

\begin{teo}\cite{GGS}\label{theoGGS2}
Let $F_k$ be a wandering component of the Fatou set. Then, the leaf space $\Sigma_k:=F_k/\mathcal{F}_k$ is a finite Riemann surface (Hausdorff). In other words, $\Sigma_k$ is a compact Riemann surface with a finite number of points deleted.
\end{teo}

\begin{ex}
The most basic examples are the linear flows of the torus $\mathbb{T}^3=\mathbb{R}^3/\mathbb{Z}^3$: the foliation in $\mathbb{R}^3$ given by parallel lines is invariant by the action of $\mathbb{Z}^3$. So it defines a foliation of $\mathbb{T}^3$ which is transversely affine and so transversely holomorphic. Changing the slope of the parallel lines, one gets a foliation with one wandering component or one semi-wandering component or one dense component. These examples are not very interesting because the Julia set is vacuous but it illustrates the three types of Fatou components of theorem \ref{theoGGS1}.

Another class of examples are the suspensions of representations $\rho:\pi_1(S)\rightarrow PSL(2,\mathbb{C})$ where $S$ is a compact Riemann surface. These examples have already been studied in the introduction. 

Many interesting examples can be found in \cite[section~8]{GGS}. The authors exhibit examples of transversely holomorphic foliations whose Fatou set consists of an arbitrary number of connected components and whose Julia set is the disjoint union of codimension $1$ submanifolds. They also exhibit examples with a Julia set having non empty interior without being the whole manifold.

\end{ex}

\section{Harmonic morphisms are Brownian path preserving.}

\paragraph{Harmonic morphisms.}

We start with some basic facts about harmonic morphisms. The reader who wants to learn more about this theory could refer, for example, to the survey of John C.Wood \cite{W} or to the book of H.Urakawa \cite{U}.

Let $(M,g)$ and $(N,h)$ be two $C^{\infty}$ Riemannian manifolds whose dimensions are respectively $m$ and $n$. Denote $\Delta$ the Laplace-Beltrami operator in $M$. A $C^{\infty}$ map $f:M\rightarrow\mathbb{R}$ satisfying $\Delta f=0$ is called a harmonic map.

\begin{df}
A $C^{\infty}$ map $\Phi:(M,g)\rightarrow (N,h)$ is a harmonic morphism if for any any open set $V\subset N$ with $\Phi^{-1}(V)\neq \varnothing$ and for any harmonic map $f:V\rightarrow\mathbb{R}$, the map  $f\circ\Phi$ is a harmonic map on $\Phi^{-1}(V)$.
\end{df}

 





\begin{df}\label{defmh}
A $C^{\infty}$ map $\Phi:(M,g)\rightarrow (N,h)$ is said to be horizontally weakly conformal if for any point $p$ in $M$ such that $D\Phi_p\neq 0$, the differential $D\Phi_p$ sends conformally the horizontal space $\ker(D\Phi_p)^{\bot}$ on the tangent space $T_{\Phi(p)}N$, in other words $D\Phi_p$ is onto and there exists a real $\lambda(p)\neq 0$ such that for all $X,Y \in \ker(D\Phi_p)^{\bot}$:
\[h_{\phi(p)}(D\Phi_p(X),D\Phi_p(Y))=\lambda(p)^2g_p(X,Y)\]
\end{df}



The following caracterisation of harmonic morphisms will be useful later:

\begin{teo}\cite{BE}\label{caractmh}
Suppose that the dimension of $N$ is $n=2$. Let $\Phi:(M,g)\rightarrow (N,h)$ be a horizontally weakly conformal map. Then $\Phi$ is a harmonic morphism if and only if the fibre of $\Phi$ over any regular point is minimal.
\end{teo}

\paragraph{Brownian motion.}

Consider a connected Riemannian manifold $(M,g)$ with bounded geometry. We denote $(B_t)_{t\geq 0}$ the Brownian motion on $M$, i.e. the diffusion process associated to the Laplace-Beltrami operator $\Delta$ on $(M,g)$. The Brownian motion is defined on the family of probability spaces $\left(\Omega_x,\mathbb{P}_x\right)_{x\in M}$ where $\Omega_x$ is the space of continuous paths $\omega:[0,\infty[\rightarrow M$ such that $\omega(0)=x$ and $\mathbb{P}_x$ is the so-called Wiener measure on $\Omega_x$. The reader could refer to \cite{CC2} for a construction of the Brownian motion on manifolds and basic facts on the subject.



In $1940$, Paul Lévy proved that a conformal map is Brownian path preserving (see \cite{Le}). The following result is a generalisation of Paul Lévy's theorem. It asserts that the maps between Riemannian manifolds which are Brownian path preserving are the harmonic morphisms. This property has been proved in the case of harmonic morphisms between Euclidean spaces in \cite{BCD}. The general case of harmonic morphisms between Riemannian manifolds has been proved later in \cite{D}. Before enonciate this theorem, let us give a definition of a Brownian path preserving map.

\begin{df}
Let $(M,g)$ and $(N,h)$ be two complete Riemannian manifolds and let $U\subset M$ be an open set. A map $\Phi:U\longrightarrow N$ is said to be Brownian path preserving if the following are satisfied for all $x\in U$ and every Brownian motion $(B_t)_{t\geq 0}$ starting at $x$:
\begin{enumerate}
\item There is a mapping $\omega\mapsto\sigma_{\omega}$ such that for every $\omega$, the map $\sigma_{\omega}:[0,T[\rightarrow[0;\infty[$ is a continuous and strictly increasing function (where $T:=\inf\{t\in[0;\infty]\text{ s.t. }B_t\notin U\}$ is the exit time of $U$). 
\item There is a Brownian motion $(B'_s)_{s\geq 0}$ starting at $\Phi(x)$ such that:
$$\Phi \circ B_t=B'_{\sigma(t)}.$$
\end{enumerate}
\end{df}

The theorem is:

\begin{teo}\cite[Theorem C]{D}\label{theobrownmh}
Let $(M,g)$ and $(N,h)$ be two $C^{\infty}$ Riemannian manifolds and $\Phi:M\longrightarrow N$ be a $C^{\infty}$ map. The map $\Phi$ is a harmonic morphism if and only if $\Phi$ is Brownian path preserving.
\end{teo}

\textbf{Remark: } if $\Phi:(M,g)\rightarrow (N,h)$ is a harmonic morphism, then one can prove that $\Phi$ is horizontally weakly conformal \cite{Fug}. So, by definition we have a dilatation coefficient $p\in M\mapsto\lambda_p\in\mathbb{R}$. And we have an explicit formula for the time changed scale $\sigma$ of the previous theorem:
$$\sigma(t)=\displaystyle{\int_0^t}\lambda^2(B_u)du.$$

\section{Proof of theorem \ref{theoperso1}.}
Let $(M,\mathcal{F})$ be a transversely holomorphic foliation on a compact manifold. Suppose that the foliation is taut. Let $F$ be a wandering component of the Fatou set and denote $\Sigma=F/\mathcal{F}$ the leaf space, which is a finite Riemann surface by theorem \ref{theoGGS2}. 

\paragraph{Step $1$: choose a good metric on $M$.}
Endow $\Sigma$ with a complete metric $h$ with constant curvature $+1$, $0$ or $-1$ in its conformal class. We have the following:

\begin{lem}
There exists a metric $g$ on $M$ satisfying:
\begin{enumerate}
\item the fibres of $p:(F,g_{|F})\longrightarrow(\Sigma,h)$ are minimal.
\item $p:(F,g_{|F})\longrightarrow(\Sigma,h)$ is horizontally conformal. In other words, for any point $x$ in $M$, the differential $Dp_x$ sends conformally the horizontal space $\ker(Dp_x)^{\bot}$ on the tangent space $T_{p(x)}N$.
\end{enumerate}
\end{lem}

\begin{proof}

Let $g_0$ be a metric on $M$ such that all the leaves of $\mathcal{F}$ are minimal. This metric has no reasons to satisfy the second item. So, we are going to modify the metric orthogonally to the leaves. For this, let $\mathcal{A}$ be an atlas with a finite number of foliated charts $\phi_{\alpha}:U_{\alpha}\rightarrow V_{\alpha}\subset \mathbb{C}$ and with transition functions $\gamma_{\alpha\beta}:\phi_{\beta}(U_{\alpha}\cap U_{\beta})\rightarrow \phi_{\alpha}(U_{\alpha}\cap U_{\beta})$. Let $u_{\alpha}$ be a partition of unity associated to this atlas. By definition, the complex structure on $\Sigma=F/\mathcal{F}$ is the one induced by the transversely holomorphic structure of the foliation. This means that $\Sigma$ is defined by the family of open sets $\phi_{\alpha}(U_{\alpha}\cap F)\subset\mathbb{C}$ which are glued together by the transition functions $\gamma_{\alpha\beta}$. So, in restriction to a chart $\phi_{\alpha}(U_{\alpha}\cap F)$, the metric $h$ is conformally equivalent to the Euclidean metric. Now, let $e_1$ and $e_2$ be two vector fields in $U_{\alpha}$ satisfying:

\begin{itemize}
\item $(\phi_{\alpha})_*(e_i)=\frac{\partial}{\partial x_i}$, for $i=1,2$
\item $e_i \in T\mathcal{F}^{\bot}$, for $i=1,2$
\end{itemize}
Complete $(e_1,e_2)$ with sections of $T\mathcal{F}_{|U_{\alpha}}$ so that, for all $x \in U_{\alpha}$, $b(x)=(e_1(x),e_2(x),e_3(x),...,e_{d+2}(x))$ is a basis of $T_x M$. We have:
\[
mat_{b(x)}g_0(x)=\begin{pmatrix}
A&0 \\
0&B
\end{pmatrix}\]
with $A=\begin{pmatrix}
a&b \\
b&d
\end{pmatrix}$.
Define the metric $g_0^{\alpha}$ in $U_{\alpha}$ by:
\[
mat_{b(x)}g_0^{\alpha}(x)=\begin{pmatrix}
C&0 \\
0&B
\end{pmatrix}\]
with $C=\begin{pmatrix}
1&0 \\
0&1
\end{pmatrix}$.
Then, write $g=\sum \limits_{\alpha}u_{\alpha}g_0^{\alpha}$.
The metric $g$ satisfies all the wanted properties: indeed we defined g so that the projection $p$ is horizontally conformal. A modification of the metric transversely to the fibres does not change the fact that these fibres are minimal.
\end{proof}

\paragraph{Step $2$: project the Brownian motion on the leaf space. }
From now on, the Brownian motion in $M$ (resp. $\Sigma$) is with respect to the metric $g$ (resp. $h$) defined in the previous lemma. We have just proved that the projection $p:(F,g_{|F})\longrightarrow(\Sigma,h)$ is horizontally conformal and that the fibres of $p$ are minimal. So, according to theorem \ref{caractmh}, $p$ is a harmonic morphism. Consequently, using theorem \ref{theobrownmh}, the map $p$ is Brownian path preserving. So, if $(B_t)_{t\geq 0}$ is the Brownian motion starting at a point $x_0$ in $F$ and stopped at the exit time $T=\inf\left\{t\geq 0\text{ s.t. }B_t \in \partial F\right\}$ of $F$, the process $p(B_{\sigma^{-1}(s)})$ is a Brownian motion in $\Sigma$ starting at $p(x_0)$ and stopped at time $\lim\limits_{t \to T}\sigma(t)$. We are going to prove that $\lim\limits_{t \to T}\sigma(t)=\infty$. It will be an easy consequence of the following:

\begin{lem}
Let $\gamma:[0,+\infty[\rightarrow M$ be a continuous path such that $\gamma(0) \in F$ and $\gamma$ hits $\partial F$ in finite time. Denote
$t_0:=\inf\{t \in [0,+\infty[\text{ s.t. } \gamma(t) \in\partial F\}$. Then $p\circ \gamma$ does not have limit when $t$ goes to $t_0$.
\end{lem}

\begin{proof}
Suppose on the contrary that $\lim\limits_{t\rightarrow t_0} p \circ \gamma(t)=z_0 \in \Sigma$. Let $U$ be a foliated chart defined in a neighborhood of $\gamma(t_0)$: the open set $U$ is then identified with $A\times B$ where $A$ is an open set in $\mathbb{R}^d$ and $B$ is an open set in $\mathbb{C}$ such that the plaques of the foliation are the sets $A\times\{z\}$. There exists $t_1<t_0$ such that for all $t\geq t_1$, we have $\gamma(t) \in U$. Let $x_0$ be a point in  $p^{-1}(\{z_0\})$ and let $X$ be a basic normal vector field such that $X(x_0)\neq 0$. Let $\widetilde{X}$ and $\tilde{Y}$ in $H^0(M,C_{\mathcal{F}}^{\epsilon\log\epsilon}(TM))$ lifting $X$ and $iX$. and consider the map:
\[\begin{array}{ccccc}
\phi & : & \mathbb{C}\times M & \to & M \\
& & (te^{i\theta},x) & \mapsto & \phi(te^{i\theta},x) \\
\end{array}\]
which associates to $(te^{i\theta},x)$ the point $\phi(te^{i\theta},x)$ of $M$ obtained by integrating the vector field $\tilde{X_{\theta}}:=\tilde{X}\cos{\theta}+\tilde{Y}\sin{\theta}$ at time $t$.
Denoting $L_{x_0}$ the leaf through $x_0$, we have two cases:
\begin{itemize}
\item  $L_{x_0}\cap \overline{U}\neq\varnothing$. Then denoting $V_r=\phi(\overline{D(0,r)} \times L_{x_0})$ and using the previous identification of $U$ and $A \times B$, we have for $r$ small enough: 

\[V_r \cap U =\coprod_{i \in I} A \times V_i\]
where the $V_i$ are closed topological discs pairwise disjoints. Remark that the $\coprod_{i \in I} A \times V_i$ does not meet the Julia set. If $W_r=p(V_r)$, then $W_r$ is a neighborhood of $z_0$. So, there exists $t_2>t_1$ such that for all $ t \in [t_2,t_0[$, we have $p \circ \gamma(t) \in W_r$. So, for $t \in [t_2,t_0[$, $\gamma(t) \in p^{-1}(W_r)\cap U = V_r \cap U =\coprod_{i \in I} A \times V_i$. So, there is $i\in I$ such that for all $t \in [t_2,t_0[$, we have $\gamma(t) \in A \times V_i$. So $\gamma(t_0)\in\overline{A\times V_i}=\overline{A}\times V_i\subset F$ which contradicts the fact that $\gamma(t_0)$ belongs to the Julia set.

\item $L_{x_0}\cap \overline{U}=\varnothing$. Using the same notations as in the previous case, one gets that for $r$ small enough, $V_r\cap U=\varnothing$. So, for $t \in [t_2,t_0[$, $\gamma(t) \in p^{-1}(W_r)\cap U = V_r \cap U=\varnothing$, which is absurd.
\end{itemize}

\end{proof}

\begin{cor}
Almost surely $\lim\limits_{t \to T}\sigma(t)=\infty$
\end{cor}

\begin{proof}
Almost surely, 
\[\begin{array}{ccccc}
B & : & [0,T[ & \to & M \\
& & t & \mapsto & B_t \\
\end{array}\]
is continuous. So, according to the previous lemma, almost surely $p\circ B_t$ does not have limit when $t$ tends to $T$. As $p$ is Brownian path preserving, there is a brownian motion $B'_s$ in $\Sigma$ such that for all $t\in[0,T[$, we have: $p \circ B_t=B'_{\sigma(t)}$. So, almost surely, $B'_{\sigma(t)}$ does not have limit when $t$ goes to $T$. So, $\lim\limits_{t \to T}\sigma(t)=\infty$
\end{proof}

\paragraph{Step $3$: Conclude.}
Denote $\nu_{x_0}$ the harmonic measure:  $\nu_{x_0}(A)=\mathbb{P}_{x_0}(B_T \in A)$ for any Borel set $A$ in $\partial F$. Recall that we assumed that $\nu_{x_0}(\partial F)=1$. We have to prove that for $\nu_{x_0}$-almost every $y \in \partial F$, for all $x \in F$, $y \in \overline{L_x}$.

As $\Sigma$ is a Riemann surface of finite type, $B_s':=(p(B_{\sigma^{-1}(s)}))_{0\leq s \leq \sigma(T)=+\infty}$ is recurrent in $\Sigma$ \cite{Gr}. Take $x \in F$, a neighborhood $U_x$ of $x$, a point $y \in supp(\nu_{x_0})$ and a neighborhood $V_y$ of $y$ in $M$. We are going to prove that $sat(U_x)\cap V_y \neq \varnothing$. As $y$ belongs to the support of $\nu_{x_0}$, we have $\nu_{x_0}(V_y)\neq 0$. Denote $A=\{\omega\in\Omega_{x_0}\ \text{ s.t. } \ B_{T(\omega)}(\omega) \in V_y\}$, we have $\mathbb{P}_{x_0}(A)=\nu_{x_0}(V_y)>0$. And so, for all $\omega \in A $, $B_t(\omega) \in V_y$ for $t$ close enough to $T(\omega)$. As $B'_s$ is recurrent in $\Sigma$, for almost every $\omega \in \Omega_{x_0}$, there is a sequence $s_n$ tending to infinity such that $B'_{s_n}(\omega) \in p(U_x)$. So, for almost every $\omega \in \Omega_{x_0}$, there is a sequence $t_n$ converging to $T(\omega)$ such that $B_{t_n}(\omega) \in p^{-1}(p(U_x))=sat(U_x)$. So, for almost every $\omega \in A$ (i.e. in a set with strictly positive probability), there is a sequence $t_n$ converging to $T(\omega)$ such that $B_{t_n}(\omega) \in sat(U_x)\cap V_y$. So, $sat(U_x)\cap V_y \neq \varnothing$.

Why does this imply that $y \in \overline{L_x}$? We have just proven that for any neighborhood $U_x$ of $x$ and any neighborhood $V_y$ of $y$, we have: $sat(U_x)\cap V_y \neq \varnothing$. So, there is a sequence $x_n$ with $x_n \rightarrow x$ and a sequence $y_n \in L_{x_n}$ with $y_n \rightarrow y$. Let $X_n$ be a sequence of basic normal vector fields in $H^0(M,C_{\mathcal{F}}(\nu^{1,0}))$ with $X_n(x_n)\neq 0$ and $\tilde{X_n}$ in $H^0(M,C_{\mathcal{F}}^{\epsilon\log\epsilon}(TM))$ lifting $X_n$ and satisfying $\Phi_n(1,x_n)=x$ where $\Phi_n: \mathbb{R} \times M \longrightarrow M$ is the flow associated to the vector field $\tilde{X_n}$. As $y_n\rightarrow y \in Julia$, $\tilde{X_n}(y_n)\rightarrow 0$. So $\Phi_n(1,y_n) \rightarrow y$. As $\Phi_n(1,.):M\longrightarrow M$ preserves the foliation, we have that $\Phi_n(1,y_n) \in L_x$. So, we have proved that $y \in \overline{L_x}$.



\end{document}